\documentclass{amsproc}                
\usepackage{tikz}
\usepackage{amsaddr}
\usepackage[autostyle]{csquotes}  
\usepackage[backend=biber]{biblatex} 
\usepackage{doi}
\usepackage[utf8]{inputenc}    
\newtheorem*{theorem*}{Main Theorem}
\newtheorem{lemma}{Lemma}
\newtheorem{corollary}[lemma]{Corollary}
\newtheorem{theorem}[lemma]{Theorem}
\newtheorem{proposition}[lemma]{Proposition}
\theoremstyle{remark}

\theoremstyle{definition}  
\newtheorem{definition}[lemma]{Definition}

\DeclareMathOperator\Inv{Inv}

\newcommand\Fq{\mathbf F_q}

\newcommand\ZZ{\mathbf Z}

\newcommand\bt[1]{\binom{#1}{2}}

\newcommand\qbin[2]{{#1\brack#2}_q}
\renewcommand\AA{\mathcal A}

\newcommand\heine{{}_2\phi_1}
\newcommand\qqpoch[2]{(q^{#1};q)_{#2}}

\title[Anti-Invariant Subspaces]{Enumeration of Anti-Invariant Subspaces and Touchard's\\ Formula for the Entries of the $q$-Hermite Catalan Matrix}
\author{Amritanshu Prasad}
\address{The Institute of Mathematical Sciences, Chennai, India.}
\address{Homi Bhabha National Institute, Mumbai, India.}
\email{amri@imsc.res.in}

\author{Samrith Ram}
\address{Indraprastha Institute of Information Technology Delhi, New Delhi, India.}
\email{samrith@gmail.com}


\DeclareFieldFormat{doi}{\href{https://dx.doi.org/#1}{\textsc{doi}}}
\DeclareFieldFormat{url}{\href{#1}{\textsc{url}}}
\renewbibmacro*{doi+eprint+url}{
  \printfield{doi}
  \newunit\newblock
  \iftoggle{bbx:eprint}{
    \usebibmacro{eprint}
  }{}
  \newunit\newblock
  \iffieldundef{doi}{
    \usebibmacro{url+urldate}}
  {}
}
\renewbibmacro{in:}{}
\AtEveryBibitem{
  \ifentrytype{book}
  {\clearfield{pages}}
}

\begin{document}

\maketitle
\begin{abstract}
  We express the number of anti-invariant subspaces for a linear operator on a finite vector space in terms of the number of its invariant subspaces.
  When the operator is diagonalizable with distinct eigenvalues, our formula gives a finite-field interpretation for the entries of the $q$-Hermite Catalan matrix.
  We also obtain an interesting new proof of Touchard's formula for these entries.
\end{abstract}
\section{Introduction}
Let $q$ be a prime power, and let $\Fq$ denote a finite field of order $q$. For nonnegative integers $n$ and $k$, let ${n \brack k}_q$ denote the $q$-binomial coefficient, which is the number of $k$-dimensional subspaces of $\Fq^n$. Recall that for a linear operator $T\in M_n(\Fq)$, a subspace $W\subset \Fq^n$ is said to be \emph{$T$-invariant} if $T(W)\subset W$. 
\begin{definition}
  For a linear operator $T\in M_n(\Fq)$, a subspace $W\subset \Fq^n$ is said to be \emph{$T$-anti-invariant} if
  \begin{displaymath}
    \dim(W + TW) = 2\dim W.
  \end{displaymath}
\end{definition}
\begin{theorem*}
  For any $T\in M_n(\Fq)$, the number of $\ell$-dimensional $T$-anti-invariant subspaces of $\Fq^n$ is given by
  \begin{equation}
    \label{eq:main}
    \omega^T_{n\ell} = q^{\binom \ell 2}\sum_{j=0}^{\ell} (-1)^j(X^T_j-X^T_{j-1})\qbin{n-\ell-j}{n-2\ell}q^{\binom{\ell-j+1}2},
  \end{equation}
  where $X^T_j$ is the number of $j$-dimensional $T$-invariant subspaces of $\Fq^n$.
\end{theorem*}
The computation of $X^T_j$ from the similarity class of $T$ as a polynomial in $q$ is easy \cite[Section~2]{https://doi.org/10.48550/arxiv.2205.11076} and has been implemented in SageMath \cite{simsage}. The formula in Eq. \eqref{eq:main} can be recast in the following form, which will be used in the proof.
\begin{equation}
  \label{eq:alt}
  \omega^{T}_{n\ell} = q^{\ell \choose 2}\sum_{j=0}^\ell(-1)^j\left({n-\ell-j\brack n-2\ell}_qq^{\ell-j+1\choose 2}+{n-\ell-j-1\brack n-2\ell}_qq^{\ell-j\choose 2}\right) X^T_j.
\end{equation}

Anti-invariant subspaces were introduced by Barr\'ia and Halmos \cite{MR748946}.
A matrix $T\in M_n(\Fq)$ is said to be \emph{$\ell$-transitive} if each $\ell\times \ell$ matrix over $\Fq$ appears as the upper left submatrix of some matrix similar to $T$. 
Using the fact that $T$ is $\ell$-transitive if and only if it admits an $\ell$-dimensional anti-invariant subspace, they characterized $\ell$-transitive matrices.
Sourour \cite{MR822138} determined the maximal dimension of an anti-invariant subspace.
Kn\"uppel and Nielsen \cite{MR2013452} defined a subspace $W$ to be $k$-fold anti-invariant if
\begin{displaymath}
  \dim(W+TW\dotsb+T^kW) = (k+1)\dim W,
\end{displaymath}
and determined the maximal dimension of a $k$-fold anti-invariant subspace for a linear operator $T$. They also showed \cite[Thm. 2.1]{MR2013452} that if $d_1\leq d_2\leq\cdots \leq d_n$ denote the degrees of the invariant factors in the Smith normal form of $xI-A$ where $A$ is the matrix $T$ with respect to some basis, then a $k$-fold anti-invariant subspace of dimension $m$ exists if and only if $d_{n-m+1}+\cdots+d_n\geq (k+1)m$.

Enumerative versions of such problems can be traced back to a paper of Bender, Coley, Robbins and Rumsey \cite{MR1141317}. In the context of pseudorandom number generation, Niederreiter~\cite{N2} made the following definition:    
given $T\in M_{md}(\Fq)$, an $m$-dimensional subspace $W\subset \Fq^{md}$ is said to be a $T$-\emph{splitting} subspace of degree $d$ if
\begin{displaymath}
  W + TW + \dotsb + T^{d-1} W = \Fq^{md}.
\end{displaymath}
 Niederreiter asked for the number of $m$-dimensional $T$-splitting subspaces of degree $d$ when the characteristic polynomial of $T$ is irreducible, unaware that the question had already been answered in \cite{MR1141317}. The number of $T$-splitting subspaces in this case is given by
\begin{displaymath} 
  q^{m(m-1)(d-1)}\frac{q^{md}-1}{q^m-1}.
\end{displaymath}
Chen and Tseng \cite{sscffa} reproved this result by developing recurrence relations involving a larger class of combinatorial problems.  
Their recurrence relations are independent of the matrix $T$, and reduce the enumeration of splitting subspaces to the enumeration of flags of $T$-invariant subspaces.
However, the recurrences are very difficult to solve in general.

Aggarwal and Ram \cite{agram2022} used the recurrences of Chen and Tseng to show that when $T$ is regular nilpotent, the number of $T$-splitting subspaces of degree $d$ is given by $q^{m^2(d-1)}$.
For regular nilpotent $T$, $X_j^T=1$ for $0\leq j\leq n$, so \eqref{eq:main} becomes
\begin{equation}
  \label{eq:nilp}
  \omega^T_{n\ell} = q^{\ell^2}\qbin{n-\ell}{n-2\ell}.
\end{equation}
Setting $\ell = m$ and $n=2m$, we recover the formula of \cite{agram2022} in the case $d=2$. Later, it came to light that these results in the regular nilpotent case follow from the results in~\cite{MR1141317}. 

Viennot's combinatorial theory of orthogonal polynomials~\cite{viennot1983theorie} places the moments of an orthogonal polynomial sequence in the first column of an infinite array known as the Catalan matrix. There is a connection between our main theorem and the Catalan matrix associated to the $q$-Hermite orthogonal polynomial sequence as defined by Ismail, Stanton and Viennot~\cite{MR930175}. This connection emerges from our work \cite{fpsac,pr} where a more general class of enumerative problems is considered.
\begin{definition}
  For a linear endomorphism $T\in M_n(\Fq)$ and a partition $\mu=(\mu_1,\dotsc,\mu_k)$ of $n$, a subspace $W\subset \Fq^n$ is said to have \emph{$T$-profile $\mu$} if
  \begin{displaymath}
    \dim(W + TW + \dotsb + T^{i-1}W) = \mu_1+\dotsb + \mu_i \text{ for each } 1\leq i\leq k.
  \end{displaymath}
  The number of subspaces with $T$-profile $\mu$ is denoted $\sigma^T_\mu$.
\end{definition}
When $n=m+\ell$, a subspace $W$ has profile $(m,\ell)$ if $W$ is $m$-dimensional, and $W+TW=\Fq^n$.
Let $T^*$ denote the transpose of $T$.
For a subspace $W\subseteq \Fq^n$, let $W^0$ denote its annihilator in the linear dual of $\Fq^n$.
When $W$ has profile $(m,\ell)$, $\dim(W^0)=l$. Also  $W+TW=\Fq^n$ if and only if $W^0\cap (TW)^0 = (W+TW)^0=\{0\}$. By using the identity
\begin{equation}
  \label{eq:dimeq}
  \dim(U\cap S^{-1}U)=2\dim U-\dim (U+ SU),
\end{equation}
for $S\in M_n(\Fq)$ and each subspace $U\subset \Fq^n$, it follows that $$\dim(W^0\cap (TW)^0)=\dim(W^0\cap (T^*)^{-1}W^0)=2 \dim W^0-\dim (W^0+ T^*W^0).$$ Therefore, $W$ has profile $(m,\ell)$ if and only if $W^0$ is an $\ell$-dimensional $T^*$-anti-invariant subspace.
Since $T$ is similar to $T^*$, we have
\begin{equation}
  \label{eq:anti-inv-to-ml}
  \omega_{n\ell}^T =\omega_{n\ell}^{T^*}= \sigma_{(m,\ell)}^T.
\end{equation}

Let $[n]$ denote the set consisting of the first $n$ positive integers. Denote by $\Pi(\mu')$ the set of partitions of $[n]$ whose block sizes are the parts of the integer partition conjugate to $\mu$.
Suppose $T$ is a diagonalizable matrix with distinct eigenvalues in $\Fq$.
One of the main results of \cite{pr} is a combinatorial formula for $\sigma_\mu^T$ in terms of a statistic $v$ on set partitions known as the \emph{interlacing number}:
\begin{equation}
  \label{eq:split-ss}
  \sigma^T_\mu = (q-1)^{\sum_{j\geq 2}\mu_j}q^{\sum_{j\geq 2}\binom{\mu_j}2}\sum_{\AA\in \Pi(\mu')}q^{v(\AA)}.
\end{equation}
From \eqref{eq:anti-inv-to-ml}, we obtain
\begin{equation}
  \label{eq:anti-inv-split-ss}
  \omega_{n\ell}^T = (q-1)^\ell q^{\binom \ell 2}a_{n,n-2\ell},
\end{equation}
where $a_{n,n-2\ell} =  \sum_\AA q^{v(\AA)}$, a sum is over partitions of $[n]$ with $\ell$ blocks of size $2$ and $n-2\ell$ singleton blocks.
In Section~\ref{sec:incomplete-chord-diags}, we show that $a_{n,n-2\ell}$ coincides with polynomials defined recursively by Touchard~\cite{MR46325} in the context of the stamp-folding problem.
Touchard \cite[Eq.~(28)]{MR46325} showed that
\begin{equation}
  \label{eq:touchard}
  (q-1)^\ell a_{n,n-2\ell} = \sum_{j=0}^\ell (-1)^j \left[\binom nj - \binom n{j-1}\right]\qbin{n-\ell-j}{n-2\ell}q^{\binom{\ell-j+1}2}.
\end{equation}

When $T$ is a diagonalizable matrix with distinct eigenvalues in $\Fq$, $X^T_j = \binom nj$.
Substituting this into the formula~\eqref{eq:main} of our main theorem gives
\begin{equation}
  \label{eq:main-split-ss}
  \omega_{n\ell}^T = q^{\binom \ell2}\sum_{j=0}^\ell (-1)^j\left[\binom nj - \binom n{j-1}\right]\qbin{n-\ell-j}{n-2\ell}q^{\binom{\ell-j+1}2}.
\end{equation}
Comparing the expressions in~\eqref{eq:anti-inv-split-ss} and~\eqref{eq:main-split-ss} gives a new, linear-algebraic proof of Touchard's formula \eqref{eq:touchard}.

\subsection*{A road map}
The proof of the main theorem consists of three parts which take up the next three sections of this article.

In the first part (Section~\ref{sec:existence}) we establish the existence of a formula of the form \eqref{eq:ajxj} for the number $\omega_{n\ell}^T$ of anti-invariant subspaces of a given dimension as a linear combination of the numbers $X_j^T$ of invariant subspaces, whose coefficients are independent of $T$.
It remains to show that the coefficients of these linear combinations are as in Eq.~(\ref{eq:alt}).

In the second part (Section~\ref{sec:proof-main-theorem}) we consider a family of matrices for which the number of anti-invariant subspaces is known.
Each matrix in this family gives rise to an equation in the coefficients of \eqref{eq:ajxj}.
Theorem~\ref{theorem:uniqueness}  establishes that this system of equations has a unique solution.
Thus in order to show that the coefficients are exactly the ones given in \eqref{eq:alt}, it suffices to show that the identity~\eqref{eq:main} holds for each matrix in the family.
This can be be expressed as the family \eqref{eq:zero-sum} of identities.

In the third part (Section~\ref{sec:proof-main-identity}) we prove the identities~\eqref{eq:zero-sum} by reducing them to Heine's transformations for $q$-hypergeometric functions.

The proof strategy in this article closely follows ideas in \cite{https://doi.org/10.48550/arxiv.2205.11076}; the main distinction is that the more general identity \eqref{eq:zero-sum} needed here requires a very different approach.

The final section (Section~\ref{sec:incomplete-chord-diags}) of this article is devoted to the case where $T$ is a diagonal matrix with distinct diagonal entries, and the connection to the Catalan matrix of $q$-Hermite orthogonal polynomials.

\section{Existence of a Universal Formula}
\label{sec:existence}
In this section we prove the existence of a universal formula for the number of anti-invariant subspaces of a given dimension for an arbitrary operator $T$. The main step is Lemma~\ref{prop:recurrence}, which is a special case of the recurrence of Chen and Tseng \cite[Lemma~2.7]{sscffa}. We begin by introducing some notation.

Given $T\in M_n(\Fq)$ and sets $A$ and $B$ of subspaces of $\Fq^n$, define
\begin{align*}
  \beta(A,B)&:=\{W \in A\mid W\cap T^{-1}W\in B\},\\
  \gamma(A,B)&:=\{(W_1,W_2)\mid W_1\in A, W_2\in B,\text{ and } W_1\cap T^{-1}W_1\supset W_2\}.
\end{align*}
For integers $a,b$, we also write $\beta(a,b)$ for the set of $a$-dimensional subspaces $W$ of $\Fq^n$ such that $W\cap T^{-1}W$ has dimension $b$. The quantity $\gamma(a,b)$ is defined analogously.
For example, using identity \eqref{eq:dimeq}, $\beta(a,a)$ denotes the set of $a$-dimensional $T$-invariant subspaces, whereas $\beta(a,0)$ denotes the set of $a$-dimensional $T$-anti-invariant subspaces. To explicitly specify the linear operator $T$, we also write $\beta^T(A,B)$ and $\gamma^T(A,B)$.
\begin{lemma}
  \label{prop:recurrence}
  For each $T\in M_n(\Fq)$ and $0\leq a\leq n$,
  \begin{align*}
    |\beta(a,b)| & = X_b^T\qbin{n-b}{a-b}-X_a^T\qbin ab\\
                 & + \sum_{j=0}^{b-1}|\beta(b,j)|{n-2b+j \brack a-2b+j}_q-\sum_{k=b+1}^{a-1}|\beta(a,k)|{k \brack b}_q.
  \end{align*}
\end{lemma}
\begin{proof}
  Since the collection of all $a$-dimensional subspaces of $\Fq^n$ is the disjoint union $ \coprod_{0\leq k\leq a} \beta(a,k)$, we have
  \begin{align*}
    \gamma(a,b)=\coprod_{0\leq k\leq a}\gamma(\beta(a,k),b).
  \end{align*}
  To count pairs of subspaces $(W_1,W_2)\in \gamma(\beta(a,k),b)$, first choose $W_1\in \beta(a,k)$ and then choose $W_2$ to be an arbitrary $b$-dimensional subspace of $W_1\cap T^{-1}W_1.$ It follows that
  \begin{align}
    |\gamma(a,b)|&=\sum_{k=0}^a |\gamma(\beta(a,k),b)|
                   =\sum_{k=b}^a|\beta(a,k)|{k \brack b}_q\nonumber\\
                 &=|\beta(a,b)|+\sum_{k=b+1}^a|\beta(a,k)|{k \brack b}_q. \label{eq:1}
  \end{align}
  Similarly, the set of all $b$-dimensional subspaces of $\Fq^n$ equals the disjoint union $\coprod_{0\leq j\leq b}\beta(b,j)$. Therefore
  \begin{align*}
    \gamma(a,b)=\coprod_{0\leq j\leq b}\gamma(a,\beta(b,j)).
  \end{align*}
  To count pairs $(W_1,W_2)\in \gamma(a,\beta(b,j)),$ first choose $W_2\in \beta(b,j)$ and note that  $\dim (W_2+TW_2)=2b-j$ by \eqref{eq:dimeq}.

  Given $W_2$, a pair $(W_1, W_2)$ belongs to $\gamma(a, \beta(b, j))$ if and only if $W_1$ is an $a$-dimensional subspace that contains $W_2 + TW_2$. Therefore, the number of choices for $W_1$ is ${n-(2b-j) \brack a-(2b-j)}_q$.
  Consequently,
  \begin{align}
    |\gamma(a,b)|&=\sum_{j=0}^b |\gamma(a,\beta(b,j))|\nonumber \\
                 &=\sum_{j=0}^b|\beta(b,j)|{n-(2b-j) \brack a-(2b-j)}_q. \label{eq:2}
  \end{align}
  The lemma now follows from Eqs.~\eqref{eq:1} and \eqref{eq:2}, and the fact that $|\beta(a,a)|=X_a^T$.
\end{proof}

\begin{proposition}\label{prop:tuple}
  Given integers $n$, $a$, $b$, there exist polynomials $p_j(t)\in \ZZ[t]$ $(0\leq j\leq a)$, such that, for every prime power $q$ and every $T\in M_n(\Fq)$,
  \begin{align*}
    |\beta^T(a,b)|=\sum_{j=0}^a p_j(q)X_j^T.
  \end{align*}
\end{proposition}
\begin{proof}
  Lemma~\ref{prop:recurrence} expands $|\beta^T(a,b)|$ in terms of $X_a^T$, $X_b^T$, and $|\beta^T(a',b')|$ where either $a'<a$, or $a'=a$ and $a'-b'<a-b$.
  The coefficients are polynomials in $q$ that are independent of $T$.
  Thus repeated application of Lemma~\ref{prop:recurrence} results in an expression of the stated form in finitely many steps.
\end{proof}

The following corollary shows the existence of a universal formula for the number of anti-invariant subspaces of a given dimension.
\begin{corollary}
  \label{cor:ajxj}
  For all integers $n\geq 2\ell\geq 0$, there exist polynomials $p_j(t)\in \ZZ[t]$ $(0\leq j\leq \ell)$ such that, for every prime power $q$ and every $T\in M_n(\Fq)$,
  \begin{equation}
    \label{eq:ajxj}
    \omega^T_{n\ell} = \sum_{j=0}^\ell p_j(q)X_j^T.
  \end{equation}
\end{corollary}
\begin{proof}
  Set $a=\ell$ and $b=0$ in Proposition \ref{prop:tuple}.
\end{proof}
\section{Determination of Coefficients in the Universal Formula}
\label{sec:proof-main-theorem}
We set up a system of linear equations which completely determine the polynomials $p_j(t)$ $(0\leq j\leq \ell)$ in Corollary \ref{cor:ajxj} by constructing, for each prime power $q$, a sequence of matrices $T_0(q),\dotsc,T_\ell(q)$ such that the following conditions are satisfied.
\begin{enumerate}
\item For each $0\leq i,j\leq \ell$, there exists a polynomial $X_{ij}(t)$ such that $X^{T_i(q)}_j=X_{ij}(q)$ for all prime powers $q$.
\item The determinant of the matrix $X(t):=(X_{ij}(t))_{0\leq i,j\leq \ell}$ is a non-zero element of $\ZZ[t]$.
\item The identity (\ref{eq:main}) holds for $T_i(q)$ for $i=0,\dotsc,\ell $ and all prime powers $q$.
\end{enumerate}
In effect, we have the following result.
\begin{theorem}
  \label{theorem:uniqueness}
  For each prime power $q$, the system of linear equations 
  \begin{equation}
    \label{eq:system}
    \omega^{T_i(q)}_{n\ell} = \sum_{j=0}^\ell p_j(q)X_{ij}(q), \quad 0\leq i\leq \ell,
  \end{equation}
  in the variables $p_0(q),\dotsc,p_\ell(q)$ has a solution given by
  \begin{align}\label{eq:soln}
    p_j(q)=(-1)^jq^{\ell \choose 2}\left({n-\ell-j\brack n-2\ell}_qq^{\ell-j+1\choose 2}+{n-\ell-j-1\brack n-2\ell}_qq^{\ell-j\choose 2}\right).
  \end{align}
  This solution is unique for sufficiently large prime powers $q$ and hence uniquely determines the polynomials $p_j(t)$ for $0\leq j\leq l$.
\end{theorem}
We now proceed with the construction of the matrices $T_i(q)(0\leq j\leq \ell)$ above. For each $q$, let $T_0(q)$ be any matrix in $M_n(\Fq)$ with irreducible characteristic polynomial.
For $i=1,\dotsc,\ell$, take $T_i(q)$ to be the $n\times n$ matrix with block decomposition
\begin{displaymath}
  T_i(q) =
  \begin{pmatrix}
    {\bf 0} &{\bf 0}\\
    {\bf 0} & T'_i(q)
  \end{pmatrix},
\end{displaymath}
where $T'_i(q)$ is a nonsingular $(\ell-i)\times (\ell-i)$ matrix with irreducible characteristic polynomial.
We have
\begin{align*}
  X_{j}^{T_i(q)}=
  \begin{cases}
    \delta_{0j}+\delta_{nj} & i=0,\\
    \qbin{n-\ell+i}j + \qbin{n-\ell+i}{j-\ell+i}  &1\leq i<\ell,\\
    \qbin nj  & i=\ell.
  \end{cases}
\end{align*} 
Therefore we can take
\begin{align}\label{eq:xijs}
  X_{ij}(t)=
  \begin{cases}
    \delta_{0j}+\delta_{nj} & i=0,\\
    {n-\ell+i \brack j}_t + {n-\ell+i \brack j-\ell+i}_t  &1\leq i<\ell,\\
    {n \brack j}_t  & i=\ell.
  \end{cases}
\end{align}

\begin{lemma}
  If $T\in M_n(\Fq)$ has irreducible characteristic polynomial, then
  \begin{equation}
    \label{eq:simple}
    \sigma^T_\mu = \frac{q^n-1}{q^{\mu_1}-1}\prod_{i\geq 2}q^{\mu_i^2-\mu_i}{\mu_{i-1} \brack \mu_i}_q.
  \end{equation}
\end{lemma}
\begin{proof}
  Follows from \cite[Prop. 4.6]{pr} and \cite[Thm. 3.3]{sscffa}.
\end{proof}
\begin{lemma}
  The equation \eqref{eq:system} holds for $i=0$.
\end{lemma}
\begin{proof}
  Take $\mu=(n-\ell,\ell)$ in (\ref{eq:simple}).
  The left hand side of (\ref{eq:system}) is given by
  \begin{displaymath}
    \omega^{T_0(q)}_{n\ell} =\sigma^{T_0(q)}_\mu= \frac{q^n-1}{q^{n-\ell}-1}q^{\ell^2-\ell}{n-\ell \brack \ell}_q.
  \end{displaymath}
  On the other hand, the right hand side of \eqref{eq:system} becomes
  \begin{align*}
    &q^{\ell \choose 2}\left( q^{\ell+1 \choose 2}{n-\ell \brack \ell}_q+q^{\ell \choose 2}{n-\ell-1 \brack \ell-1}_q\right)\\
    &=q^{\ell^2-\ell}{n-\ell \brack \ell}_q\left(q^\ell+\frac{[\ell]_q}{[n-\ell]_q}\right)\\
    &=q^{\ell^2-\ell}{n-\ell \brack \ell}_q\frac{[n]_q}{[n-\ell]_q} = \frac{q^n-1}{q^{n-\ell}-1}q^{\ell^2-\ell}\qbin{n-\ell}{n-2\ell},
  \end{align*}
  establishing \eqref{eq:system} for $i=0$.
\end{proof}
It remains to show that the values of $p_j(q)$ given by \eqref{eq:soln} are solutions to \eqref{eq:system} for $1\leq i\leq \ell$. We begin by showing that the left hand side of \eqref{eq:system} vanishes in these cases.
\begin{lemma}\label{lem:vanishes}
  For $1\leq i\leq \ell $, we have $\omega^{T_i(q)}_{n\ell} = 0$.
\end{lemma}
\begin{proof}
  The $\Fq[t]$-module $M_i=\Fq^n$, where $t$ acts by $T_i(q)$ is of the form
  \begin{displaymath}
    M_i =
    \begin{cases}
      \Fq^{n-\ell+i}\oplus \Fq[t]/f_i(t) &  i=1,\dotsc,\ell-1,\\
      \Fq^n & i=\ell.
    \end{cases}
  \end{displaymath}
  Here $f_i(t)$ denotes the characteristic polynomial of $T'_i(q)$ for $i=1,\dotsc,\ell-1$.
  Since none of the modules $M_i$ $(1\leq i\leq \ell)$ can be generated by $n-\ell+1$ or fewer generators, $T_i(q)$ does not admit a subspace with profile $(n-\ell,\ell)$.
  In other words, $\omega^{T_i(q)}_{n\ell} = \sigma^{T_i(q)}_{(n-\ell,\ell)}=0$ for $i=1,\dotsc,\ell $.
\end{proof}
In view of  Lemma \ref{lem:vanishes}, in order to establish \eqref{eq:system} for $T_i(q)$ $(1\leq i\leq \ell)$, it suffices to prove the identity
\begin{displaymath}
  \sum_{j=0}^\ell (-1)^j (X_{ij}(q)-X_{i,j-1}(q))\qbin{n-\ell-j}{n-2\ell}q^{\binom{\ell-j+1}2}=0.
\end{displaymath}
Let
\begin{displaymath}
  Y(n,\ell,i,j)= {n-\ell+i \brack j}_q+{n-\ell+i \brack j-\ell+i}_q- {n-\ell+i \brack j-1}_q-{n-\ell+i \brack j-1-\ell+i}_q.
\end{displaymath}
Then by Eq.~\eqref{eq:xijs},
\begin{displaymath}
  Y(n,\ell,i,j) = \begin{cases}
    X_{ij}(q)-X_{i,j-1}(q) & \text{for }1\leq i<\ell\\
    2(X_{ij}(q)-X_{i,j-1}(q))& \text{for } i=\ell.
  \end{cases}
\end{displaymath}
Therefore, it suffices to show that, for $1\leq i\leq \ell$,
\begin{equation}
  \label{eq:zero-sum}
  \sum_{j=0}^\ell (-1)^j Y(n,\ell,i,j)\qbin{n-\ell-j}{n-2\ell}q^{\binom{\ell-j+1}2}=0.
\end{equation}
The proof uses techniques from the theory of $q$-hypergeometric series, and appears in Section~\ref{sec:proof-main-identity}.

The non-singularity of $X(t)=(X_{ij}(t))_{0\leq i,j\leq \ell}$ is proved by using inequalities satisfied by the degrees of its entries.
We recall \cite[Lemma 4.4]{https://doi.org/10.48550/arxiv.2205.11076}.
\begin{lemma}
  \label{lem:maxperm}
  Let $(a_{ij})_{n\times n}$ be a real matrix such that whenever $i<k$ and $j<k$,
  $$
  a_{ i k}-a_{ij}< a_{kk}-a_{kj}.
  $$
  Then the sum $S(\sigma)=\sum_{1\leq i\leq n}a_{i\sigma(i)}$ attains its maximum value precisely when $\sigma$ is the identity permutation.
\end{lemma}
\begin{proposition}
  \label{prop:cofactor}
  For all $n\geq 2l\geq 0$, the determinant of the matrix $X(t)$ is non-zero.
  Therefore, for sufficiently large prime powers $q$, $\det X(q)\neq 0$.
\end{proposition}
\begin{proof}
  By Eq.~\eqref{eq:xijs}, the first row of $X(t)$ is the unit vector $(1,0,\dotsc,0)$.
  Therefore it suffices to show that the determinant of the submatrix $X'(t)=(X_{ij}(t))_{1\leq i,j\leq \ell}$ is non-zero.
  Let $a_{ij}=\deg X_{ij}(t)$.
  Since $\deg {n\brack k}_t = k(n-k)$, it follows from \eqref{eq:xijs} that, for $1\leq i,j\leq \ell$,
  \begin{displaymath}
    a_{ij}=\max\{j(n-\ell+i-j),(j-\ell+i)(n-j)\}=j(n-\ell+i-j),
  \end{displaymath}
  since
  \begin{align*}
    j(n-\ell+i-j)-(j-\ell+i)(n-j)=(\ell-i)(n-2j)\geq 0.
  \end{align*}
  If $i<k$ and $j<k$, then
  \begin{align*}
    a_{ik}-a_{ij}&=k(n-\ell+i-k)-j(n-\ell+i-j)\\
                 &=(k-j)(n-\ell+i-k-j)\\
                 &<(k-j)(n-\ell-j)\\
                 &=a_{kk}-a_{kj}.
  \end{align*}
  Now Lemma \ref{lem:maxperm} implies that $\det X'(t)$ has degree $\sum_{i=1}^\ell a_{ii}>0$ and is thus non-zero.
\end{proof}
This completes all steps in the proof of Theorem \ref{theorem:uniqueness} except for the identity \eqref{eq:zero-sum}. 
\section{Reduction to Heine's Transformations}
In this section, we prove the identity \eqref{eq:zero-sum} encountered in the proof of Theorem~\ref{theorem:uniqueness} by using a Heine transformation for $q$-hypergeometric series.
\label{sec:proof-main-identity}
Accordingly, define
\begin{align*}
  Y_1(n,\ell,i,j):= {n-\ell+i \brack j}_q,\quad&    Y_2(n,\ell,i,j):= {n-\ell+i \brack j-\ell+i}_q,\\
  Y_3(n,\ell,i,j):={n-\ell+i \brack j-1}_q,\quad&  Y_4(n,\ell,i,j):=    {n-\ell+i \brack j-1-\ell+i}_q.
\end{align*}
Let $S_r(n,\ell,i)$ denote the sum obtained in \eqref{eq:zero-sum} by replacing $Y(n,\ell,i,j)$ by $Y_r(n,\ell,i,j)$.
We need to show that $S_1+S_2-S_3-S_4=0$.
We will show that $S_1=S_4$ while $S_2=S_3$ by expressing the sums $S_i$ as $q$-hypergeometric series.

Define, as usual, the $q$-Pochhammer symbols
\begin{displaymath}
  (a;q)_\infty = \prod_{k=0}^\infty (1-aq^k) \text{ and } (a;q)_n = \frac{(a;q)_\infty}{(aq^n;q)_\infty} = \prod_{k=0}^{n-1}(1-aq^k).
\end{displaymath}
For convenience, we will also use the notation
\begin{displaymath}
  (a,b;q)_n = (a;q)_n(b;q)_n.
\end{displaymath}
Heine (see Gasper and Rahman \cite{MR2128719}) defined the $q$-hypergeometric series
\begin{displaymath}
  \heine(a,b;c;q,z) = \sum_{n\geq 0} \frac{(a,b;q)_n}{(q,c;q)_n}z^n.
\end{displaymath}
\begin{lemma}\label{lem:foursums}
  Let $m=n-\ell$.
  For all $1\leq i\leq \ell$, we have
  \begin{align}
    \tag{$S1$}\label{eq:s1}
    S_1(n,\ell,i) & = (-1)^\ell\frac{(q^{m+i-\ell+1};q)_\ell}{(q;q)_\ell}\;\heine(q^{-\ell},q^{m-\ell+1};q^{m+i-\ell+1};q,q^{\ell+1}),\\
    \tag{$S2$}\label{eq:s2}
    S_2(n,\ell,i) & = (-1)^\ell\frac{(q^{m+1};q)_i}{(q;q)_i}\;\heine(q^{-i},q^{m-\ell+1};q^{m+1};q,q^{i+1}),\\
    \tag{$S3$}\label{eq:s3}
    S_3(n,\ell,i) & = (-1)^\ell\frac{(q^{m+i-\ell+2};q)_{\ell-1}}{(q;q)_{\ell-1}}\;\heine(q^{1-\ell},q^{m-\ell+1};q^{m+i-\ell+2};q,q^{\ell}),\\
    \tag{$S4$}\label{eq:s4}
    S_4(n,\ell,i) & = (-1)^\ell\frac{(q^{m+2};q)_{i-1}}{(q;q)_{i-1}}\;\heine(q^{1-i},q^{m-\ell+1};q^{m+2};q,q^{i}).
  \end{align}
\end{lemma}
We will see that Heine's transformation formula \cite[Eq. (III.2)]{MR2128719}
transforms $S_1(n,\ell,i)$ into $S_4(n,\ell,i)$ and $S_2(n,\ell,i)$ into $S_3(n,\ell,i)$; therefore \eqref{eq:zero-sum} will follow from the lemma.
\begin{proof}
  [Proof of the lemma]
  We will use the following identities (equation numbers refer to Gasper and Rahman \cite[Appendix~I]{MR2128719})
  \begin{gather}
    \tag{I.7}\label{eq:I.7}
    (a;q)_n = (q^{1-n}/a;q)_n(-a)^n q^{\binom n2},\\
    \tag{I.10}\label{eq:I.10}
    (a;q)_{n-k} = \frac{(a;q)_n}{(q^{1-n}/a;q)_k}\left(-\frac qa\right)^kq^{\binom k2-nk},\\
    \tag{I.42}\label{eq:I.42}
    \qbin\alpha k = \frac{(q^{-\alpha};q)_k}{(q;q)_k}(-q^\alpha)^kq^{-\binom k2}.
  \end{gather}
  Replacing $j$ by $\ell-j$ in the sum
  \begin{displaymath}
    S_1(n,\ell,i) = \sum_{j=0}^\ell(-1)^j \qbin{m+i}j\qbin{m-j}{m-\ell}q^{\binom{\ell-j+1}2}
  \end{displaymath}
  gives
  \begin{displaymath}
    \sum_{j=0}^\ell (-1)^{\ell-j}\qbin{m+i}{\ell-j}\qbin{m-\ell+j}j q^{\binom{j+1}2}.
  \end{displaymath}
  The identity~\eqref{eq:I.42} allows us to write
  \begin{displaymath}
    \qbin{m+i}{\ell-j}\qbin{m-\ell+j}j = (-1)^\ell\frac{\qqpoch{-m-i}{\ell-j}}{\qqpoch{}{\ell-j}}\frac{\qqpoch{\ell-m-j}{j}}{\qqpoch{}{j}}q^{\ell(m+i)-\bt \ell-ij}.
  \end{displaymath}
  Applying~(\ref{eq:I.10}) to $\qqpoch{-m-i}{\ell-j}$ and $\qqpoch{}{\ell-j}$ gives
  \begin{displaymath}
    (-1)^\ell\frac{\qqpoch{-m-i}{\ell}}{\qqpoch{}{\ell}}\frac{\qqpoch{-\ell}{j}}{\qqpoch{m+i-\ell+1}{j}}\frac{\qqpoch{\ell-m-j}{j}}{\qqpoch{}{j}}q^{\ell(m+i)+(m+1)j-\bt \ell}.
  \end{displaymath}
  Applying~(\ref{eq:I.7}) to $\qqpoch{-m-i}{\ell}$ and $\qqpoch{\ell-m-j}{j}$ gives
  \begin{equation}
    \label{eq:s1sum}
    \qbin{m+i}{\ell-j}\qbin{m-\ell+j}j = (-1)^j\frac{\qqpoch{m+i-\ell+1}{\ell}}{\qqpoch{}{\ell}}
    \frac{(q^{-\ell},q^{m-\ell+1};q)_j}{(q,q^{m+i-\ell+1};q)_j}q^{\ell j-\bt j}.
  \end{equation}
  Thus
  \begin{displaymath}
    S_1(n,\ell,i) = (-1)^\ell\frac{\qqpoch{m+i-\ell+1}{\ell}}{\qqpoch{}{\ell}}\sum_{j=0}^\ell\frac{(q^{-\ell},q^{m-\ell+1};q)_j}{(q,q^{m+i-\ell+1};q)_j}q^{(\ell+1)j}.
  \end{displaymath}
  Since $(q^{-\ell};q)_j=0$ for $j>\ell$, the sum can be extended to infinity, giving (\ref{eq:s1}).

  In order to prove \eqref{eq:s3}, observe that
  \begin{align*}
    \qbin{m+i}{\ell-j-1}\qbin{m-\ell+j}j = \qbin{m+i}{\ell-j}\qbin{m-\ell+j}j \frac{1-q^{\ell-j}}{1-q^{m+i-\ell+j+1}}.
  \end{align*}
  Apply~(\ref{eq:s1sum}) and rewrite the right hand side as
  \begin{displaymath}
    (-1)^j\frac{\qqpoch{m+i-\ell+1}\ell}{\qqpoch{}\ell}\frac{(q^{-\ell},q^{m-\ell+1};q)_j}{(q,q^{m+i-\ell+1};q)_j}\frac{1-q^{\ell-j}}{1-q^{m+i-\ell+j+1}}q^{\ell j-\bt j}.
  \end{displaymath}
  Making the substitutions
  \begin{align*}
    \frac{\qqpoch{m+i-\ell+1}\ell}{\qqpoch{}\ell} & = \frac{\qqpoch{m+i-\ell+2}{\ell-1}}{\qqpoch{}{\ell-1}}\frac{1-q^{m+i-\ell+1}}{1-q^\ell},\\
    \frac{\qqpoch{-\ell}j}{\qqpoch{m+i-\ell+1}j} & = \frac{1-q^{-\ell}}{1-q^{j-\ell}}\frac{\qqpoch{1-\ell}j}{\qqpoch{m+i-\ell+2}j}\frac{1-q^{m+i-\ell+j+1}}{1-q^{m+i-\ell+1}},
  \end{align*}
  and cancelling out common factors gives
  \begin{equation}
    \label{eq:s3sum}
    \qbin{m+i}{\ell-j-1}\qbin{m-\ell+j}j = (-1)^j\frac{\qqpoch{m+i-\ell+2}{\ell-1}}{\qqpoch{}{\ell-1}}\frac{(q^{1-\ell},q^{m-\ell+1};q)_j}{(q,q^{m+i-\ell+2};q)_j}q^{(\ell-1)j-\bt j}.
  \end{equation}
  Evaluating $S_3(n,\ell,i)$ after replacing $j$ by $\ell-j$ in the sum and using \eqref{eq:s3sum} gives \eqref{eq:s3}.

  To prove (\ref{eq:s4}), we proceed as in the proof of (\ref{eq:s1}).
  \begin{align*}
    S_4(n,\ell,i) & = \sum_{j=0}^\ell\qbin{m+i}{j-1-\ell+i}\qbin{m-j}{m-\ell}q^{\binom{\ell-j+1}2}\\
                  & = \sum_{j=0}^\ell (-1)^{\ell-j}\qbin{m+i}{i-1-j}\qbin{m-\ell+j}j q^{\binom{j+1}2}.
  \end{align*}
  The identity \eqref{eq:I.42} allows us to write
  \begin{multline*}
    \qbin{m+i}{i-1-j}\qbin{m-\ell+j}j \\= (-1)^{i-1}\frac{\qqpoch{-m-i}{i-1-j}}{\qqpoch{}{i-1-j}}\frac{\qqpoch{\ell-m-j}j}{\qqpoch{}j}q^{(m+1)(i-1)+\bt i - (\ell+1)j}.
  \end{multline*}
  Applying \eqref{eq:I.10} to $\qqpoch{-m-i}{i-1-j}$ and $\qqpoch{}{i-1-j}$ gives
  \begin{displaymath}
    (-1)^{i-1}\frac{\qqpoch{-m-i}{i-1}}{\qqpoch{}{i-1}}\frac{\qqpoch{1-i}j}{\qqpoch{m+2}j}\frac{\qqpoch{\ell-m-j}j}{\qqpoch{}j}q^{(m+1)(i-1)+\bt i + (m+i-\ell)j}.
  \end{displaymath}
  Finally, applying \eqref{eq:I.7} to the terms $\qqpoch{-m-i}{i-1}$ and $\qqpoch{\ell-m-j}j$ gives
  \begin{equation}
    \label{eq:s4sum}
    \qbin{m+i}{i-1-j}\qbin{m-\ell+j}j = (-1)^j\frac{\qqpoch{m+2}{i-1}}{\qqpoch{}{i-1}}\frac{(q^{1-i},q^{m-\ell+1};q)_j}{(q,q^{m+2};q)_j}q^{-\bt{j+1}+ij}
  \end{equation}
  Evaluating $S_4(n,\ell,i)$ using \eqref{eq:s4sum} gives \eqref{eq:s4}.

  In order to prove \eqref{eq:s2}, observe that
  \begin{displaymath}
    \qbin{m+i}{i-j} \qbin{m-\ell+j}j = \qbin{m+i}{i-j-1}\qbin{m-\ell+j}j \frac{1-q^{m+j+1}}{1-q^{i-j}}.
  \end{displaymath}
  Applying \eqref{eq:s4sum} allows us to write the right hand side as
  \begin{displaymath}
    (-1)^j\frac{\qqpoch{m+2}{i-1}}{\qqpoch{}{i-1}}\frac{(q^{1-i},q^{m-\ell+1};q)_j}{(q,q^{m+2};q)_j}\frac{1-q^{m+j+1}}{1-q^{i-j}}q^{-\bt{j+1}+ij}.
  \end{displaymath}
  Making the substitutions
  \begin{align*}
    \frac{\qqpoch{m+2}{i-1}}{\qqpoch{}{i-1}} & = \frac{\qqpoch{m+1}i}{\qqpoch{}i}\frac{1-q^i}{1-q^{m+1}},\\
    \frac{\qqpoch{1-i}j}{\qqpoch{m+2}j} & = \frac{1-q^{j-i}}{1-q^{-i}}\frac{\qqpoch{-i}j}{\qqpoch{m+1}j}\frac{1-q^{m+1}}{1-q^{m+j+1}}
  \end{align*}
  and cancelling out common factors gives
  \begin{equation}
    \label{eq:s2sum}
    \qbin{m+i}{i-j} \qbin{m-\ell+j}j = (-1)^j\frac{\qqpoch{m+1}i}{\qqpoch{}i}\frac{(q^{-i},q^{m-\ell+1};q)_j}{(q,q^{m+1};q)_j}q^{-\bt{j+1}+(i+1)j}.
  \end{equation}
  Evaluating $S_2(n,\ell,i)$ using \eqref{eq:s2sum} after replacing $j$ by $\ell-j$ in the sum gives (\ref{eq:s2}).
\end{proof}
\begin{proof}
  [Proof that $S_1=S_4$]
  Applying Heine's transformation formula \cite[Eq.~(III.2)]{MR2128719}
  \begin{equation}
    \label{eq:Heine}
    \heine(a,b;c;q,z) = \frac{(c/b,bz;q)_\infty}{(c,z;q)_\infty}\;\heine(abz/c,b;bz;q,c/b)
  \end{equation}
  with $(a,b,c,z)=(q^{-\ell},q^{m-\ell+1},q^{m+i-\ell+1},q^{\ell+1})$ gives
  \begin{displaymath}
    S_1(n,\ell,i) =  (-1)^\ell\frac{(q^{m+i-\ell+1};q)_\ell}{(q;q)_\ell}\frac{(q^i,q^{m+2};q)_\infty}{(q^{m+i-\ell+1},q^{\ell+1};q)_\infty}\heine(q^{1-i},q^{m-\ell+1};q^{m+2};q,q^i).
  \end{displaymath}
  Observe that
  \begin{align*}
    \frac{(q^{m+i-\ell+1};q)_\ell}{(q;q)_\ell}\frac{(q^i,q^{m+2};q)_\infty}{(q^{m+i-\ell+1},q^{\ell+1};q)_\infty} & = \frac{(q^{m+i-\ell+1};q)_\ell}{(q;q)_\ell}\frac{\qqpoch{i}{\ell-i+1}}{\qqpoch{m+i-\ell+1}{\ell-i+1}}\\
                                                                                                                  &=\frac{\qqpoch{m+2}{i-1}}{\qqpoch{}{i-1}},
  \end{align*}
  which gives $S_1(n,\ell,i)=S_4(n,\ell,i)$.
\end{proof}
\begin{proof}
  [Proof that $S_2=S_3$]
  Heine's transformation formula (\ref{eq:Heine}) with $(a,b,c,z)=(q^{1-\ell},q^{m-\ell+1},q^{m+i-\ell+2},q^\ell)$ gives
  \begin{displaymath}
    S_3(n,\ell,i) = (-1)^\ell\frac{\qqpoch{m+i-\ell+2}{\ell-1}}{\qqpoch{}{\ell-1}}\frac{(q^{i+1},q^{m+1};q)_\infty}{(q^{m+i-\ell+2},q^\ell;q)_\infty}\heine(q^{-i},q^{m-\ell+1};q^{m+1};q,q^{i+1}).
  \end{displaymath}
  Observe that
  \begin{align*}
    \frac{\qqpoch{m+i-\ell+2}{\ell-1}}{\qqpoch{}{\ell-1}}\frac{(q^{i+1},q^{m+1};q)_\infty}{(q^{m+i-\ell+2},q^\ell;q)_\infty} & = \frac{\qqpoch{m+i-\ell+2}{\ell-1}}{\qqpoch{}{\ell-1}}\frac{\qqpoch{i+1}{\ell-i-1}}{\qqpoch{m+i-\ell+2}{\ell-i-1}}\\
                                                                                                                             & = \frac{\qqpoch{m+1}i}{\qqpoch{}i},
  \end{align*}
  giving $S_2(n,\ell,i)=S_3(n,\ell,i)$.
\end{proof}
This completes the proof of the main theorem.
\section{The $q$-Hermite Catalan matrix}
\label{sec:incomplete-chord-diags}
In this section, we discuss the connection between our main theorem and Touchard's formula for the entries of the $q$-Hermite Catalan matrix. 

An \emph{extended chord diagram} is a visual representation of an involution $\sigma$ on $[n]$.
Arrange $n$ nodes labelled $1,\dotsc,n$ along the $X$-axis.
To their right, add a node labelled $\infty$.
A circular arc lying above the $X$-axis is used to connect the elements of each $2$-cycle of $\sigma$.
Each fixed point of $\sigma$ is connected to the node $\infty$.
The extended chord diagram of the involution $(1,4)(2,6)(7,8)$ on the set $[8]$ is shown below:
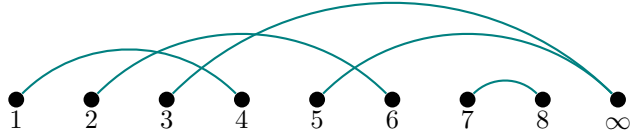
\begin{center}  
  \begin{tikzpicture}
    [every node/.style={circle,fill=black,inner sep=0pt, minimum size=6pt}]
    \node[label=below:$1$] (1) at (1,0) {};
    \node[label=below:$2$] (2) at (2,0) {};
    \node[label=below:$3$] (3) at (3,0) {};
    \node[label=below:$4$] (4) at (4,0) {};
    \node[label=below:$5$] (5) at (5,0) {};
    \node[label=below:$6$] (6) at (6,0) {};
    \node[label=below:$7$] (7) at (7,0) {};
    \node[label=below:$8$] (8) at (8,0) {};
    \node[label=below:$\infty$] (0) at (9,0) {};
    \draw[thick,color=teal]
    (1) [out=45, in=135] to  (4);
    \draw[thick,color=teal]
    (2) [out=45, in=135] to  (6);
    \draw[thick,color=teal]
    (3) [out=45, in=135] to  (0);
    \draw[thick,color=teal]
    (5) [out=45, in=135] to  (0);
    \draw[thick,color=teal]
    (7) [out=45, in=135] to  (8);
  \end{tikzpicture}
\end{center}
A \emph{crossing} is a pair of arcs $[(i,j),(k,\ell)]$ such that $i<k<j<\ell$.
The extended chord diagram above has four crossings, namely $[(1,4),(2,6)]$, $[(1,4),(3,\infty)]$, $[(2,6),(3,\infty)]$, and $[(2,6),(5,\infty)]$.
Let $v(\sigma)$ denote the number of crossings of the extended chord diagram of an involution $\sigma$.

Let $\Inv(n,k)$ denote the set of involutions in $S_n$ with $k$ fixed points.
Define
\begin{equation}
  \label{eq:ank}
  a_{nk}(q) = \sum_{\sigma\in \Inv(n,k)} q^{v(\sigma)}.
\end{equation}
If $n-k$ is odd then $\Inv(n,k)=\emptyset$, so $a_{nk}(q)=0$.
If $n-k$ is even, then an element of $\Inv(n,k)$ has $l:=(n-k)/2$ cycles of length two.

For each non-negative integer $k$, let $[k]_q$ denote the $q$-integer $1+q+\dotsb+ q^{k-1}$.
\begin{lemma}
  \label{lemma:Catalan}
  We have
  \begin{gather*}
    a_{00}(q) = 1, \quad a_{0k} = 0 \text{ for } k>0,\\
    a_{nk}(q) = a_{n-1,k-1}(q) + [k+1]_q a_{n-1,k+1}(q) \quad (n>0).
  \end{gather*}
\end{lemma}
\begin{proof}
  Each involution $\sigma\in \Inv(n-1,k-1)$ can be extended to an element of $\Inv(n,k)$ by adding $n$ as a fixed point.
  Furthermore, each $\sigma\in \Inv(n-1,k+1)$ can be extended to an element of $\Inv(n,k)$ in $k+1$ different ways: any one of its $k+1$ fixed points can be paired with $n$.
  Pairing the $r$-th fixed point from right to left with $n$ results in $r-1$ new crossings.
  Taken together, these $k+1$ choices contribute $(1+q+\dotsb +q^k)q^{v(\sigma)}=[k+1]_q q^{v(\sigma)}$ to $a_{nk}(q)$.
  Since every element of $\Inv(n,k)$ can be obtained uniquely by one of these methods, the identity of the lemma follows.
\end{proof}
Touchard \cite{MR46325} studied the polynomials
\begin{displaymath}
  T_m(q) = a_{2m,0}(q) = \sum_{\sigma\in \Inv(2m,0)} q^{v(\sigma)},
\end{displaymath}
which admit a simple (but subtle) formula
\begin{equation}
  \label{eq:tr}
  (q-1)^m T_m(q) = \sum_{j=0}^{m} (-1)^j\left[\binom{2m}j - \binom{2m}{j-1}\right] q^{\binom{m-j+1}2},
\end{equation}
known as the \emph{Touchard-Riordan formula}.
The polynomials $a_{n,n-2\ell}$ are precisely the entries of the Catalan matrix associated to the normalized $q$-Hermite orthogonal polynomials of Ismail, Stanton and Viennot~\cite{MR930175}.
See Aigner~\cite[Chapter~7]{MR2339282} for a comprehensive exposition.

The combinatorial theory of orthogonal polynomials \cite{viennot1983theorie} places the moments of an orthogonal polynomial sequence in the first column of a Catalan matrix: if an orthogonal polynomial sequence $\{P_k(x)\}_{k\geq 0}$ satisfies the three-term recurrence relation
\begin{gather*}
  P_{k+1}(x) = (x-b_k)P_k(x)-\lambda_kP_{k-1}(x) \text{ for } k\geq 1,\\
  \text{with }P_0(x)=1,\quad P_1(x)=x-b_0,
\end{gather*}
for some $\{b_k\}_{k\geq 0}$ and $\{\lambda_k\}_{k\geq 1}$, with $\lambda_k\neq 0$, the entries of the Catalan matrix $(c_{nk})_{n,k\geq 0}$ are given by
\begin{gather*}
  c_{00} = 1, \quad c_{0k} = 0 \text{ for } k>0,\\
  c_{nk} = c_{n-1,k-1} + b_k c_{n-1,k} + \lambda_{k+1} c_{n-1,k+1}.
\end{gather*}
The moments of the orthogonal polynomial sequence are
\begin{displaymath}
  \mu_n = c_{n0} \text{ for } n\geq 0.
\end{displaymath}
Lemma~\ref{lemma:Catalan} implies that the polynomials $a_{nk}(q)$ of (\ref{eq:ank}) are the entries of the Catalan matrix with $b_k=0$ and $\lambda_k=[k]_q$, which correspond to the combinatorial version of the $q$-Hermite orthogonal polynomial sequence \cite[Eq.~(2.11)]{MR930175}.
Thus Touchard's polynomials $T_m(q)$ are the even moments of the $q$-Hermite orthogonal polynomial sequence (the odd moments being $0$).
This is well-known and plays a role in the proof of the Touchard-Riordan formula \cite[Chapter~7]{MR2339282}.

Let $T$ be an $n\times n$ diagonalizable matrix with distinct eigenvalues in $\Fq$.
By (\ref{eq:anti-inv-to-ml}), the number of $\ell$-dimensional $T$-anti-invariant subspaces is equal to the number of subspaces with $T$-profile $(m,\ell)$, where $m=n-\ell$.
The conjugate of the partition $\mu=(m,\ell)$ is the partition $\mu'=(2^\ell,1^{m-\ell})$.
Set partitions with $\ell$ blocks of size $2$ and $m-\ell$ blocks of size $1$ can be identified with involutions on $[n]$ with $m-\ell$ fixed points.
The interlacing number \cite[Defn.~3.3]{pr} of such a set partition reduces to the number of crossings of the extended chord diagram on the corresponding involution.
Therefore, the formula (\ref{eq:split-ss}) can be rephrased as follows.
\begin{theorem}
  \label{theorem:rss}
  For all integers $n\geq 2\ell\geq 0$,
  \begin{displaymath}
    \omega_{n\ell}^T = (q-1)^\ell q^{\binom \ell 2}a_{n,n-2\ell}(q).
  \end{displaymath}
\end{theorem}
Since all $T$-invariant subspaces are direct sums of eigenspaces, $X^T_j$ is just the binomial coefficient $\binom nj$.
Combining Theorem~\ref{theorem:rss} with our formula (\ref{eq:main}) for $\omega^T_{n\ell}$ gives a new proof of Touchard's formula \cite[Eq.~(28)]{MR46325} for $a_{n,n-2\ell}(q)$. 
\begin{theorem}[Touchard's formula]
  \label{theorem:Touchard}
  For all integers $n\geq 2\ell\geq 0$,
  \begin{displaymath}
    (q-1)^\ell a_{n,n-2\ell}(q) = \sum_{j=0}^\ell (-1)^j\left[\binom nj-\binom n{j-1}\right]\qbin{n-\ell-j}{n-2\ell}q^{\binom{\ell-j+1}2}.
  \end{displaymath}
\end{theorem}
Specializing to $n=2\ell$ recovers the Touchard-Riordan formula \eqref{eq:tr}.

\section{Acknowledgements}
We are indebted to an anonymous referee for several comments and suggestions that helped improve the overall presentation of this paper. We thank Divya Aggarwal for her comments on an earlier draft of this manuscript.
We thank Michael Schlosser for suggesting the method used in the proof of the main identity \eqref{eq:zero-sum}.
The second author was partially supported by a MATRICS grant MTR/2017/000794 awarded by the Science and Engineering Research Board and an Indo-Russian project DST/INT/RUS/RSF/P41/2021.

\printbibliography
  
\end{document}    
\end{displaymath}